\documentclass[12pt]{article}
\usepackage[a4paper, left=2.5cm, right=2.5cm]{geometry}
\usepackage{graphicx}
\usepackage[utf8]{inputenc}
\usepackage{color}
\usepackage{latexsym}
\usepackage{amsbsy}
\usepackage{amsthm}
\usepackage{amssymb}
\usepackage{amsmath}
\usepackage{amsfonts}
\usepackage{dsfont}
\usepackage{enumerate}
\usepackage{hyperref}
\hypersetup{
	colorlinks=true,
	linkcolor=black,
	linkbordercolor={1 1 1},
	citecolor=black
}

\newcommand{\IE}{\mathds{E}}
\newcommand{\IN}{\mathds{N}}
\newcommand{\IP}{\mathds{P}}
\newcommand{\IR}{\mathds{R}}
\newcommand{\Sp}{\mathds{S}}

\newcommand{\diam}{\mathrm{diam}}
\DeclareMathOperator{\bey}{bey}
\DeclareMathOperator{\Int}{int}
\DeclareMathOperator{\relint}{relint}

\DeclareMathOperator{\Proj}{proj}

\def\E{{\mathds E}}

\def\1{{\mathds 1}}

\def\N{{\mathds N}}
\def\P{{\mathds{P}}}

\def\V{{\mathds V}}
\def\R{{\mathds{R}}}

\def\cB{{\mathcal B}}
\def\cC{{\mathcal C}}

\def\cF{{\mathcal F}}

\def\cH{{\mathcal H}}

\def\cN{{\mathcal N}}

\def\cT{{\mathcal T}}

\def\cX{{\mathcal X}}
\def\cY{{\mathcal Y}}

\def\e{\varepsilon}

\def\k{{\kappa}}

\def\r{{\rho}}

\def\bd{{\partial}}

\newcommand{\dd}{\mathrm{d}}

\newtheorem{theorem}{Theorem}[section]

\newtheorem{lemma}[theorem]{Lemma}

\newtheorem{proposition}[theorem]{Proposition}

\theoremstyle{definition}
\newtheorem{remark}[theorem]{Remark}

\begin{document}

\title{Central limit theorems for random boundary polytopes}
\author{
Matthias Reitzner\thanks{Institut f\"ur Mathematik, Universit\"at Osnabr\"uck, Germany; matthias.reitzner@uni-osnabrueck.de} and
Mathias Sonnleitner\thanks{Institut f\"ur Finanzmathematik und Angewandte Zahlentheorie, Johannes Kepler Universit\"at Linz, Austria; mathias.sonnleitner@jku.at}}
\date{}

\maketitle

\begin{abstract}
	The number of faces of the convex hull of $n$ independent and identically distributed random points chosen on the boundary of a smooth convex body in $ \R^d$ is investigated. In dimensions two and three the number of $k$-faces is known to be constant almost surely and in dimension four and higher the variance is known to be non-zero if $k\ge 1$. We show that it is of order $n$. This is complemented by a central limit theorem with a Berry-Esseen bound which is of optimal order $n^{-1/2}$.  We derive similar results for the Poissonized model, where additionally the number of random points is Poisson distributed. As a main tool, we develop a representation of the number of faces as a sum of exponentially stabilizing score functions. 
	\\

	\noindent{\bf Keywords}. random polytope, random approximation, stochastic geometry, variance expansion, Berry-Esseen bound, central limit theorem\\
	{\bf MSC 2020}. Primary: 52A22; Secondary: 60D05.
\end{abstract}

\maketitle

\section{Introduction and main results}

Let $K$ be a smooth convex body in $\R^d$. Choose $n$ independent and identically distributed random points $X_1, \dots, X_n$ on the boundary $\bd K$ of $K$, and denote by $P_n= [X_1, \dots, X_n]$ the convex hull of these points. We are interested in the expectation of the number of $k$-faces $\E f_k (P_n)$ of $P_n$, their variance $\V f_k(P_n)$ and their fluctuations, where $k =1, \dots, d-1$. Since explicit results for fixed $n$ cannot be expected, we investigate the asymptotics as $n \to \infty$.

There is a vast amount of literature on random polytopes with vertices chosen from the \emph{interior} of a convex body, i.e. a compact convex set with non-empty interior. Investigations started with two classical articles by R\'enyi and Sulanke \cite{RS63, RS64} dealing with polygons and smooth convex bodies in the two-dimensional case. Their results for smooth convex bodies were generalized to higher dimensions by Wieacker \cite{Wi}, Schneider and Wieacker \cite{SchWi1}, B\'ar\'any \cite{Bar2}, Sch\"utt \cite{Schu2} and B\"or\"oczky, Hoffmann and Hug \cite{BHH}, and to all intrinsic volumes by B\'ar\'any \cite{Bar2, Bar2c} and Reitzner \cite{Re04}. It turns out that the difference between the intrinsic volume of $K$ and the random polytope with uniformly distributed points in the interior of $K$ is asymptotically of order $n^{- \frac 2 {d+1}}$.  The corresponding results for polytopes are much more difficult. The results by R\'enyi and Sulanke for polygons were generalized in a long and intricate proof by B\'ar\'any and Buchta \cite{BB93}, from which it follows that the order of the volume difference is $n^{-1} (\ln n)^{d-1} $.

Due to Efron's identity, \cite{Ef}, the results on the volume difference can be used to determine the expected number $\E f_0(P_n)$ of vertices of $P_n$. More general, asymptotic expressions for the number of $k$-dimensional faces $f_k (P_n)$ are due to Wieacker \cite{Wi}, B\'ar\'any and Buchta \cite{BB93} and Reitzner \cite{Re05}.
In recent years the results concerning expectations of functionals of the random polytope have been accomplished by variance estimates and central limit theorems, see 
\cite{BR10a, BR10, BS13, CSY13, CY14, CY17, Gro88, Gry21, GRT23, LachSchulteYukich, Pardon1, Pardon2, Re03, Re05, ThTuWe18}.

Summarizing the results on expectations and variances for the $f$-vectors we have
\begin{eqnarray*}
\E f_k  (P_n) &=& c_{d,k} \, \Omega(K) \, n^{ \frac {d-1} {d+1}}(1+o(1)),
\\
\V f_k  (P_n) &=& c_{d,k} \, \Omega(K) \, n^{ \frac {d-1} {d+1}}(1+o(1)),
\end{eqnarray*}
for $k  \in \{0, \dots, d-1\}$, if $K$ is a smooth convex body and the random polytope $P_n$ is the convex hull of $n$ independent uniformly distributed random points in the interior of $K$.
(Throughout this paper, we denote by $c_{d,k}$ constants only depending on $d$ or $k$ and by $c,C,c_0,\dots$ constants which may additionally depend on $K$ and whose precise values may differ from line to line.)
And if $K$ is a polytope, then
\begin{eqnarray*}
\E f_k  (P_n) &=&  c_{d,k}\,   \operatorname{flag}(P) \, (\ln n)^{ {d-1}} (1+o(1)) ,
\\
\V f_k  (P_n) &=&  c_{d,k}\,   \operatorname{flag}(P) \, (\ln n)^{ {d-1}} (1+o(1)) ,
\end{eqnarray*}
where ${\rm flag}(P)$ is the number of flags of $P$. The last equality is known to hold for simple polytopes $K$ whereas for general polytopes $K$ only the exact order $(\ln n)^{ {d-1}}$ is known. 

\medskip
In this paper, we are investigating the case that the points $X_1,\dots,X_n$ are uniformly distributed on the boundary of a smooth convex body $K$ equipped with $(d-1)$-dimensional Hausdorff measure $\cH^{d-1}$. Here, a convex body is smooth, $K \in \cC^2_+$, if its boundary is a twice continuously differentiable manifold with everywhere positive Gaussian curvature $\k(x)$ for $x\in \bd K$. Taking the convex hull of the random points gives a random polytope $P_n$ where, almost surely, all $n$ random points are vertices,
$$ f_0(P_n)=n ,$$
because $K$ is strictly convex. In dimension two, this implies $f_1(P_n)=n$ and thus the combinatorial structure of $P_n$ is deterministic. This also holds in dimension three, where by Euler's polytope formula, almost surely, 
$$f_0(P_n) = n,\ f_1(P_n)= 3n-6,\ f_2(P_n) = 2n-4 $$
holds. Clearly, this shows that $\V f_k(P_n)=0$ for $k=0$ and for all $k$ in dimension $d \leq 3$.
It was observed by Stemeseder \cite{Stemeseder} that for $d \geq 4$ the random variable $f_k (P_n)$ is not deterministic for $1 \leq k  \leq d-1$. In his (unpublished) PhD thesis he proved that
\begin{equation} \label{eq:exp-faces}
\E \, f_{k } (P_n) = c_{d,k } n(1+ o(1)),
\end{equation}
and thus, perhaps surprisingly, the asymptotic behavior to first order is independent of $K$. The constant $c_{d,k}$ in \eqref{eq:exp-faces} can be expressed in terms of the expected internal angles of random simplices inscribed into the unit sphere, see Kablucko, Thäle and Zaporozhets \cite[Remark 1.9]{KTZ20}, and is known explicitly, see \cite{Kab21, Kab23}. In fact, for $X_1,\dots,X_n$ uniformly distributed on the unit sphere, the expectation $\E\,f_k(P_n)$ itself is known by \cite[Thm. 1.3]{KTZ20} but even in this simpler case the variance $\V f_k(P_n)$ appears to be unknown for $d\ge 4$.

For general $K\in \cC^2_+$, an upper bound for the variance follows immediately from the Efron-Stein jackknife inequality and Theorem 10 in \cite{Re03}, 
\begin{equation} \label{eq:var-upper}
\V f_k (P_n) \leq n \E (f_k (P_{n+1})-f_k (P_n))^2  \leq C n .
\end{equation}
A lower bound for the variance is much more involved since the variance vanishes for $k=0$ and $d \leq 3$, and it is a priori unclear whether the variance asymptotically vanishes also in other cases.
It is the aim of this contribution to prove the missing lower bound for the variance.

\begin{theorem}\label{thm:var}
	Let $d\ge 4$ and let $X_1,\dots,X_n$ be independently and uniformly distributed on the boundary of a smooth $d$-dimensional convex body $K\in \mathcal{C}^2_+$ and $P_n=[X_1,\dots,X_n]$. Then there exist constants $c>0$ and $n_0\in \N$ such that, for every $n\ge n_0$ and $k=1,\dots,d-1$,
$$
\V f_k(P_n)\ge c n.
$$
\end{theorem}

This variance lower bound is one essential step to prove a central limit theorem. The second ingredient is the exponential decay of the radius of stabilization, which will be defined and discussed in Section~\ref{sec:clt}. Combined with a quantitative central limit theorem by Lachieze-Rey, Schulte and Yukich \cite{LachSchulteYukich} for exponentially stabilizing functionals, which is derived by Stein's method, this will yield the following central limit theorem for the face numbers of $P_n$, which is novel even in the special case of points uniformly distributed on the unit sphere.

\begin{theorem}\label{thm:CLT}
	Let $d\ge 4$ and $n\ge d+2$. Let $X_1,\dots,X_n$ be independently and uniformly distributed on the boundary of a smooth $d$-dimensional convex body $K\in \cC_{2}^+$ and $P_n=[X_1,\dots,X_n]$. For every $k=1,\dots,d-1$, the rescaled fluctuations of $f_k(P_n)$ around $\E f_k(P_n)$ converge in distribution to a standard Gaussian variable $Z \sim \cN(0,1)$ and 
$$
d_K \Big(\frac{f_k(P_n)-\E f_k(P_n)}{\sqrt{\V f_k(P_n)}},Z\Big)
\le \frac{C}{\sqrt{n}} ,
$$
where $d_K(\cdot,Z)=\sup_{t\in \IR}|\IP(\cdot\le t)-\IP(Z\le t)|$ denotes the Kolmogorov distance to $Z$. 
\end{theorem}

\begin{remark}
Since the resulting rate of convergence in Theorem~\ref{thm:CLT} is of the same order as $(\V  f_k(P_n))^{-1/2}$ and since $f_k(P_n)$ is an integer-valued random variable, an argument of Englund \cite[Section 6]{Englund} shows that, up to numerical constants, our result is in fact optimal.
\end{remark}

\begin{remark}
	All results generalize easily to the case when the random points are sampled on $\partial K$ according to a continuous positive density $h$ with respect to $\cH^{d-1}$. In this case, the constants also depend on the positive numbers $\min_{x\in \partial K}h(x)$ and $\max_{x\in \partial K}h(x)$.
\end{remark}

\begin{remark}
For a different reason, a change in behavior at $d=4$ occurs also for the order of magnitude of the expected minimal area of facets of random polytopes inscribed into the sphere, see Leroux, Rademacher, Schütt and Werner \cite{LRS+25}. Yet another such threshold at $d=4$ was observed for the shape of fluctuations of the total surface area of a Poisson point process of hyperbolic hyperplanes, see Herold, Hug, Kabluchko, Rosen and Thäle \cite{HHT21,KRT25}.
\end{remark}

Corresponding results for the volume difference $V_d(K)-V_d(P_n)$ and the intrinsic volumes $V_i(K)-V_i(P_n)$ have already been obtained. First contributions are due to Buchta, M\"uller and Tichy \cite{BMT85}, Schneider \cite{Schn88}, Sch\"utt and Werner \cite{SW5}, Reitzner \cite{Re02} and B\"or\"oczky, Fodor, and Hug \cite{BFH13} who investigated the expectation of the volume difference, respectively intrinsic volume difference.  Bounds for the variance and central limit theorems are due to Reitzner \cite{Re03}, Richardson, Vu, and Wu \cite{RVW08}, Th\"ale \cite{Thaele18} and Turchi and Wespi \cite{TuWe}. Generalizations to other geometries are by Besau, Rosen and Thäle \cite{BRT21}.

If the underlying convex body $K$ is a polytope and the random polytope $P_n$ is the convex hull of uniform random points on the boundary, the situation is much more complicated. It is immediate that in this case even $f_0(P_n)$ is a random variable, and only recently the expectation of $f_0(P_n), f_{d-1}(P_n)$ and $V_d(K)-V_d(P_n)$ have been obtained by Reitzner, Sch\"utt and Werner \cite{RSW23}. The other face numbers, intrinsic volumes, variance asymptotics or central limit theorems seem to be out of reach at the moment. Even less is known for the Hausdorff distance between $K$ and $P_n$, see Prochno, Sch\"utt, Sonnleitner and Werner \cite{PSS+25}.

\medskip
We want to remark that there is a modified model including a further randomization, where instead of taking $n$ random points, the number of points $N$ is Poisson distributed with mean $t$. Then $X_1,\dots,X_N$ is distributed like a Poisson point process. In particular, restrictions to disjoint parts are then independent from each other. This additional Poissonization turned out to be a powerful method to derive variance bounds and limit theorems for the random polytope model with uniform points inside the convex body $K$, and de-Poissonization methods yielded limit theorems for fixed $n=t$. For our model with points on $\bd K$, this Poissonized model was investigated by Richardson, Vu, and Wu \cite{RVW08} and Stemeseder \cite{Stemeseder} for the volume difference between $K$ and $P_N$.  This method with subsequent de-Poissonization will not work in our case because the additional variance from Poissonization is of the same order as the variance of $f_k(P_n)$ itself. This rules out the possiblity to prove limit theorems for the random polytopes generated by random points on $\bd K$ with the same method as for random polytopes generated by uniform points in $K$. This seems to be the reason why there are no results for $f_k(P_N)$ even in the Poisson setting. 

As a consequence of our results, we obtain variance bounds and a central limit theorem for the Poisson boundary polytope. 

\begin{theorem}\label{thm:poisson}
Let $d\ge 4$ and $N\sim\pi(t)$ be Poisson distributed with parameter $t>1$. Let $X_1,\dots,X_N$ be independently and uniformly distributed on the boundary of a smooth $d$-dimensional convex body $K\in \cC_{2}^+$ and $P_N=[X_1,\dots,X_N]$. There exists $t_0>1$ such that, for every $t\ge t_0$ and $k=1,\dots,d-1$,
\begin{equation}\label{eq:Poisson-Var}
c\, t \le \V f_k(P_N)\le C\, t
\end{equation}
and 
\begin{equation}\label{eq:Poisson-CLT}
d_K \Big(\frac{f_k(P_N)-\E f_k(P_N)}{\sqrt{\V f_k(P_N)}},Z\Big)\le \frac{C}{\sqrt{t}} .
\end{equation}
\end{theorem}

\bigskip

The paper is organized in the following way. The next section contains background on the combinatorial structure of polytopes and the geometry of smooth convex bodies. In Section~\ref{sec:var} we give a proof of Theorem~\ref{thm:var}. Section~\ref{sec:clt} contains the proof of Theorem~\ref{thm:CLT} and in Section~\ref{sec:poisson} we prove Theorem~\ref{thm:poisson}.

\section{Background material}\label{sec:prelim}

\subsection{Polytopes with {$d+2$} vertices} \label{sec:polytope}

We introduce some basic notation. For a $d$-dimensional polytope $P$ we denote by $\cF_k(P)$ the set of its $k$-dimensional faces and write $f_k(P)$ for the cardinality of $\cF_k(P)$. The $f$-vector of $P$ is given by $(f_0(P),\dots,f_{d-1}(P))$. The affine span of a facet $F\in \cF_{d-1}(P)$ is an affine hyperplane $H_F$ and we denote by $H_F^+$ the corresponding closed half-space meeting $P$ only in its boundary and by $H_F^-$ the closed half-space containing $P$. 

To construct polytopes with few vertices, the beyond-beneath construction can be employed. A comprehensive account of this method and convex polytopes in general is Grünbaum's book \cite{Gruenbaum}. A point $x\in \IR^d$ is said to be beyond $F$ if it lies in $\Int H_F^+$ and beneath $F$ if it is in $\Int H_F^-$. Equivalently, $x$ is beyond $F$ if it sees $\relint F$ (the relative interior of $F$ is taken with respect to $H_F$), that is, if for any $y\in \relint F$ the line segment $[x,y]$ does not intersect $\Int P$.

The combinatorial type classifies polytopes according to their combinatorial structure. Two polytopes have the same combinatorial type if there exists an inclusion-preserving one-to-one correspondence between their sets of faces. In particular, then their $f$-vectors are identical. Trivially, the convex hull of $d+1$ affinely independent points is always a $d$-simplex and of one type. Adding one more point gives rise to $\lfloor \frac d2 \rfloor$ combinatorial types $\cT_j^d$, $j=1, \dots, \lfloor \frac d2  \rfloor$, satisfying
$$
f_k(\cT_j^d)
=
\binom{d+2}{d-k+1} - \underbrace{\binom{j+1}{d-k+1}}_{ k\geq d-j}  - \underbrace{\binom{d-j+1}{d-k+1}}_{k \geq j},\qquad k=1,\dots,d-1.
$$
In particular, for $d\leq 3$ there is only one type $\cT^d_1$ but in dimension $d \geq 4$ there exists $\cT^d_2$ with
\begin{equation} \label{eq:types}
f_k(\cT_1^d) < f_k(\cT_2^d),\qquad k =1, \dots, d-1.
\end{equation}
There is a precise description of these types.  A polytope $P$ with $d+2$ vertices is the convex hull of a $d$-simplex $S$ and a point $x\in\IR^d\setminus S$ and is of type $\cT^d_j$ for
\[
j=\bey(x,S):=|\{F\in \cF_{d-1}(S)\colon x \text{ is beyond }F\}|\in \{1,\dots,d-1\},
\]
if $j\le \lfloor \frac{d}{2}\rfloor$, and of type $\cT^d_{d-j}$, otherwise. Note that by Theorem 5.2.1 in \cite{Gruenbaum}, $\bey(x,S)=d$ implies $P=S$, which is not possible, and that $\bey(x,S)=0$ if and only if $x\in S$.  Given a $d$-simplex $S$ the regions
\begin{equation} \label{eq:beyond}
S(j)=\{x\in\IR^d\setminus S\colon \bey(x,S)=j\},\qquad j=1,\dots,d-1,
\end{equation}
partition $\IR^d\setminus S$ into polyhedral open sets whose closures cover $\IR^d\setminus S$ and which meet at the vertices of $S$. Given $x\in \IR^d\setminus S$, the type of $[S,x]$ is determined by the region $S(j)$ it belongs to (and undetermined if $x$ belongs to multiple regions). 

\begin{remark}
Almost surely, the random polytopes we consider are simplicial, that is, all of their faces are simplices. By the Dehn-Sommerville equations, which generalize Euler's polytope formula, the $f$-vectors of simplicial polytopes span an affine space of dimension $\lfloor \frac{1}{2}d\rfloor$. Thus, fixing the number of vertices leaves no degree of freedom in dimensions three or less but one degree in dimensions four and five, and so on. 
\end{remark}

\subsection{Smooth convex bodies}\label{sec:smooth}

Let $K\in \cC^2_+$, where $\cC^2_+$ is the set of convex bodies in $\IR^d$ with twice continuously differentiable boundary and everywhere positive Gaussian curvature. In the following, we collect implications of the assumption $K\in \cC^2_+$ on the surface area ($d-1$-dimensional Hausdorff measure) of caps and balls in $\partial K$.

For a unit vector $u\in \mathbb{S}^{d-1}$ and $t\in \IR$, we denote by $H(u,t)$ the affine hyperplane $\{x\in\IR^d\colon  \langle x,u\rangle=t\}$ and by $H^+(u,t)$ the corresponding halfspace $\{x\in\IR^d\colon  \langle x,u\rangle\ge t\}$. The support function $h_K(u)$ of $K$ at $u$ is given by $h_K(u)=\sup_{x\in K}\langle x,u\rangle$ and $H(u,h_K(u))$ is a supporting hyperplane.  Since $K\in \cC^2_+$, the intersection of $K$ with $H^+(u,h_K(u))$ is a single point $y$ on the boundary $\bd K$ and $u_y=u$ is the outward unit normal of $K$ at $y$. Denote by
\[
C^{\bd K}(y,h)
=\bd K\cap H^+(u_y,h_K(u_y)-h)
\]
a non-solid cap of $K$ with center $y\in \bd K$ and height $h>0$. Replacing $\bd K$ by $K$ we obtain the corresponding solid cap. For the proof of Theorem~\ref{thm:var}, we need the following construction of pairwise disjoint caps on $\bd K$ of equal height and comparable surface area. 

\begin{lemma}\label{lem:covering}
Let $K\in \cC^2_+$. There are constants $c,C>0$ such that for any $n\in\IN$ we find pairwise disjoint caps $C^{\bd K}(y_1,h_n),\dots,C^{\partial K}(y_n,h_n)$	with $y_1,\dots,y_n\in \partial K$ and 
\[
c\, n^{-\frac{2}{d-1}}
\le h_n
\le C\, n^{-\frac{2}{d-1}}.
\]
Moreover,
\[
c\, n^{-1}\le \cH^{d-1}\big(C^{\bd K}(y_i,h_n)\big)\le C\, n^{-1}.
\]
\end{lemma}

Lemma~\ref{lem:covering} is derived from the construction in Lemma 6 in \cite{Reitz05} and the following well-known asymptotics for the surface area of (non-solid) caps of smooth convex bodies, see e.g.~\cite[Lemma~6.2]{RVW08} and \cite[Lemma~5]{Reitz05}.

\begin{lemma}\label{lem:cap-convex}
Let $K\in \cC_+^2$. Then, for any $x\in \bd K$ and $h\le \diam(K)$, 
\[
c\, h^{\frac{d-1}{2}}
\le \cH^{d-1}\big(C^{\bd K}(x,h)\big)
\le C\, h^{\frac{d-1}{2}},
\]
and for $r\le \diam(K)$,
\begin{equation*} 
c \, r^{d-1}
\le \cH^{d-1}\big(B(x,r)\cap \bd K\big)
\le C\, r^{d-1},
\end{equation*}
where $B(x,r)$ denotes a Euclidean ball with center $x$ and radius $r$, $\diam(K)$ is the diameter of $K$, and $c,C>0$ only depend on $K$.
\end{lemma}

The estimates in Lemma~\ref{lem:cap-convex} can be proven via Blaschke's rolling theorem, see e.g.~\cite[Cor.~3.2.13]{Sch14}. It states that there are positive radii $r_{in}=r_{in}(K)$ and $ r_{out}=r_{out}(K)$ such that for each boundary point $x \in \bd K$ there are balls of radius $r_{in}$, respectively $r_{out}$, touching $\bd K$ at $x$ from the inside, respectively outside. That is,
\begin{equation} \label{eq:in-out}
B(x-r_{in} u_x, r_{in}) 
\subset K
\subset B(x-r_{out} u_x, r_{out}).
\end{equation}
The following consequence of Blaschke's rolling theorem is used in the proof of Theorem~\ref{thm:CLT}.

\begin{lemma} \label{lem:surface-balls}
Let $K\in \cC_+^2$. Then, for any $x\in \bd K$ and $r>0$, 
\begin{equation*}
\limsup_{\e \to 0} \frac{\cH^{d-1}\big(B(x, r+\e)\cap \bd K\big) - \cH^{d-1}\big(B(x,r)\cap \bd K\big)}{\e}
\leq
C r^{d-2},
\end{equation*}
where $C>0$ only depends on $K$.
\end{lemma}
\begin{proof}
Since $K$ is bounded, it is sufficient to show the statement for $r\le \frac{r_{in}}{4}$. Without loss of generality assume that $x=0$ and $u_x=-e_d$. Then \eqref{eq:in-out} becomes
\begin{equation} \label{eq:in-out-0}
B_{in}:=B(r_{in} e_d, r_{in})
\subset K
\subset B(r_{out} e_d, r_{out})=:B_{out}.
\end{equation}
We can parametrize $\bd K\cap B(0,\frac{r_{in}}{2})$ by a twice continuously differentiable convex function $f\colon \IR^{d-1}\to \IR$ such that 
\begin{equation*} 
\bd K\cap B\Big(0,\frac{r_{in}}{2}\Big)
=\{ (u,f(u))\colon u\in U\},\qquad U:=\Proj\Big(\bd K\cap B\Big(0,\frac{r_{in}}{2}\Big)\Big),
\end{equation*}
where $\Proj$ denotes the orthogonal projection onto $H(e_d,0)$, identified with $\IR^{d-1}$. For any $y\in \partial K\cap B(0,\frac{r_{in}}{2})$, the corresponding supporting hyperplane intersects the lower-dimensional ball $H(e_d,0)\cap B(0,\frac{r_{in}}{2})$ but not the interior of $B_{in}$. Therefore, 
\begin{equation*} 
\sup_{x\in U}\sqrt{1+\|\nabla f(x)\|^2} \le C,
\end{equation*}
where $C>0$ is an absolute constant. For $\e<\min\{\frac{r_{in}}{4},1\}$, the annular region 
\[
A_{r,\varepsilon}
:=\Proj\big(\big(B(0,r+\e)\setminus B(0,r)\big)\cap \bd K\big)
\]
is contained in $U$ and thus
\begin{align}\label{eq:annulus-bound}
	\cH^{d-1}\big( (B(x, r+\e)\setminus B(x,r))\cap \bd K\big)
	&=\int_{A_{r,\varepsilon}}\sqrt{1+\|\nabla f(x)\|^2}\dd x\notag\\
	&\le C \cH^{d-1}(A_{r,\varepsilon}).
\end{align}

We use spherical coordinates to parametrize $A_{r,\varepsilon}$. For any unit vector $u\in \Sp^{d-2}$ and $s>0$, let $\alpha_{u,s}>0$ be the unique number such that 
\[
\big(\alpha_{u,s}u,f(\alpha_{u,s}u)\big)\in \partial K\cap \partial B(0,s)
\]
and thus
$
A_{r,\varepsilon}
=\{\alpha_{u,s}u\colon r< s \le r+\varepsilon\}.
$
By \eqref{eq:in-out-0} it holds that, for each $u\in \Sp^{d-2}$ and any $s,t>0$ such that $t+s\le r_{in}$, 
\begin{equation*} 
	\alpha_{u,s}\le \alpha_{u,s+t} \le \alpha_{u,s}+t.
\end{equation*}
In particular, $\alpha_{u,s}\le s$. Therefore,
\begin{align*}
	\cH^{d-1}(A_{r,\varepsilon})
	&=\int_{\Sp^{d-2}}\int_{\alpha_{u,r}}^{\alpha_{u,r+\varepsilon}}s^{d-2}\dd s\, \dd\cH^{d-1}(u)\\
	&\le \frac{1}{d-1}\int_{\Sp^{d-2}}\big[(\alpha_{u,r}+\varepsilon)^{d-1}-\alpha_{u,r}^{d-1}\big]\dd\cH^{d-1}(u)\\
	&\le \cH^{d-1}(\Sp^{d-2})\big(\varepsilon r^{d-2}+\varepsilon^2(1+r)^{d-1}\big).
\end{align*}
Combined with \eqref{eq:annulus-bound} this completes the proof of Lemma~\ref{lem:surface-balls}.
\end{proof}

\section{Proof of the lower bound on the variance}\label{sec:var}

In this section we prove Theorem~\ref{thm:var}. Let $K\in \cC^2_+$ and assume without loss of generality that $\cH^{d-1}(\bd K)=1$. Inspired by the proof of the lower bound on the variance in \cite[Theorem~4]{Reitz05}, we use Lemma~\ref{lem:covering} to obtain a covering by $n$ caps of area roughly $n^{-1}$ and then locally approximate $\partial K$ by osculating paraboloids. We require the following construction of proportional subsets in each small cap. Recall that for a $d$-simplex $S$ we denote by $S(j)$ the sets in \eqref{eq:beyond}, which partition the exterior of $S$ into polyhedral sets meeting at the vertices of $S$.

\begin{proposition}\label{pro:lower-main}
	There exists $h_0>0$ such that the following holds. Let $C:=C^{\partial K}(y,h)$ be a cap with center $y\in \bd K$ and height $h\le h_0$. Then there exist pairwise disjoint sets $C^0,\dots,C^d\subset C$ with
\begin{equation} \label{eq:area-smallcaps}
	\cH^{d-1}(C^i)\ge c\,  \cH^{d-1}(C), \quad i=0,\dots,d.
\end{equation}
Moreover, there is a set $\tilde{C}^0\subset C$ containing $C^0$ such that for any choice of $x_0\in C^0,\dots, x_d\in C^d$, the simplex $\Delta_x=[x_0,\dots,x_d]$ satisfies 
\begin{equation} \label{eq:area-types}
	\cH^{d-1}\big(\tilde{C}^0\cap \Delta_x(j)\big)\ge c\, \cH^{d-1}(\tilde{C}^0),\qquad j=1,\dots,d-1,
\end{equation}
and  the relative boundary of $C$ is contained in the affine convex cone generated by $\Delta_x$ and apex $y$, that is,
\begin{equation} \label{eq:simplex-cone-affine}
	K\cap H(u_y,h_K(u_y)-h)\subset \{y+\lambda (z-y)\colon  z\in \Delta_x,\lambda\in \IR_+\}.
\end{equation}
The constant $c>0$ depends only on $K$.
\end{proposition}

Before we prove Proposition~\ref{pro:lower-main}, let us show how it implies the lower bound on the variance.

\begin{proof}[Proof of Theorem~\ref{thm:var}]
	According to Lemma~\ref{lem:covering} choose pairwise disjoint caps $C_j:=C^{\partial K}(y_j,h_n)$ with centers $y_1,\dots,y_n\in \bd K$, height $h_n$ of order $n^{-\frac{2}{d-1}}$ and surface area of order $n^{-1}$. Assume that $n$ is sufficiently large such that $h_n\le h_0$ for $h_0$ as in Proposition~\ref{pro:lower-main}. In each cap $C_j$, we obtain sets $\tilde{C}^0_j,C^0_j,\dots,C^d_j\subset C_j$, each occupying at least a constant fraction of $C_j$ in terms of surface area,  and for every choice of $x_0\in C^0_j,\dots, x_d\in C^d_j$, it holds that, for $k=1,\dots,d-1$,
\[
\cH^{d-1}(\{x\in \tilde{C}_j^0\colon \bey(x,[x_0,\dots,x_d])=k\})
\ge c \,\cH^{d-1}(\tilde{C}_j^0),
\]
where we recall that $c>0$ is a constant depending only on $K$ which may change from line to line in the remainder of the proof.  Thus, if $Y$ is uniformly distributed on $\tilde{C}_j^0$, then the random polytope $[Y,x_0,\dots,x_d]$ can have type $\cT^d_1,\dots,\cT^d_{\lfloor \frac{d}{2}\rfloor}$, each with positive probability. In particular, by \eqref{eq:types},
\[
\V_Y f_k([Y,x_0,\dots,x_{d}])\ge c>0,\qquad k=1,\dots,d-1.
\]

Consider the random points $X_1,\dots,X_n$ and their convex hull $P_n=[X_1,\dots,X_n]$ as in the statement of Theorem~\ref{thm:var}. For each $j=1,\dots,n$, let $A_j$ be the event that exactly one of these random points is contained in each $C^i_j$, $i=0,\dots,d$, one additional random point in $\tilde{C}_j^0$, and no other of the random points in $C_j$. Then
\begin{align*}
\IP(A_j)
&=\frac{n!}{(n-d-2)!}\IP(X_i\in C^{i}_j,0\le i\le d)\times\IP(X_{d+1}\in \tilde{C}_j^0)\times\IP(X_{\ell}\not\in C_j,\ell\ge d+2)\\
&=\frac{n!}{(n-d-2)!}\prod_{i=0}^d\cH^{d-1}(C^i_j)\times \cH^{d-1}(\tilde{C}_j^0)\times(1-\cH^{d-1}(C_j))^{n-d-2},
\end{align*}
and it follows from the estimates in Proposition~\ref{pro:lower-main} that $\IP(A_j)\ge c>0$.  Denote by $I(A_j)$ the indicator function of $A_j$. Then
\[
\IE\sum_{i=1}^{n}I(A_j)
= \sum_{i=1}^{n}\IP(A_j)
\ge c n.
\]
Further, denote by $\cF$ the position of all $X_1,\dots,X_n$ except those which are contained in caps $C_j$ with $I(A_j)=1$. Then, for $k=1,\dots,d-1$, 
\begin{align} \label{eq:var-lower}
\V f_k(P_n)
&= \IE\big[\V[f_k(P_n)|\cF]\big]+ \V\big[\IE[f_k(P_n)|\cF]\big]\notag\\
&\ge \IE\big[\V[f_k(P_n)|\cF]\big].
\end{align}
Assume that $I(A_i)=I(A_j)=1$ for some $i,j\in \{1,\dots,n\}$ and without loss of generality that $X_i$, respectively $X_j$, is one of the at most two points in $C^0_i$, respectively $C^0_j$. By \eqref{eq:simplex-cone-affine}, see \cite{Reitz05}, there is no edge between $X_i$ and $X_j$. Hence, moving $X_i$ does not change the number of faces of $P_n$ incident with $X_j$, and vice versa. This independence structure of $P_n$ implies, for every $k=1,\dots,d-1$,
\[
\IE\big[\V[f_k(P_n)|\cF]\big]
=\IE\sum_{I(A_j)=1}\V_{X_j}f_k(P_n)
\ge C \IE\sum_{I(A_j)=1}I(A_j)
\ge c n,
\]
which together with \eqref{eq:var-lower} completes the proof of Theorem~\ref{thm:var}.
\end{proof}

It remains to prove Proposition~\ref{pro:lower-main}. To this end, we shall use ideas from the proof of \cite[Theorem~1]{TuWe} which follows \cite{Reitz05,RVW08}. In particular, it follows from \cite[Section~3.1]{TuWe} that on any cap $C=C^K(y,h)$ with $y\in \partial K$ and $h$ small enough there exist pairwise disjoint sets $\tilde C^0,\dots,\tilde C^d$ on $C$ such that \eqref{eq:area-smallcaps} holds and for any choice of $x_0\in \tilde C^0,\dots,x_d\in\tilde C^d$ the inclusion \eqref{eq:simplex-cone-affine} holds. Thus, the main part of the proof of Proposition~\ref{pro:lower-main} is to show that also \eqref{eq:area-types} holds after suitably shrinking the above sets $\tilde C^i$. The idea is to fix $\tilde{C}^0=C^0$ and then subsequently choose $C^i \subset \tilde C^i$ proportionally small enough to satisfy \eqref{eq:area-types}.

\begin{proof}[Proof of Proposition~\ref{pro:lower-main}]
Let $E$ be the standard paraboloid in $\IR^d$,
\[
E=\bigg\{x\in \IR^d\colon  \langle x,e_d\rangle \ge \sum_{j=1}^{d-1}\langle x,e_j\rangle^2\bigg\}.
\]
Inscribe a $d$-simplex $S_0=[v_0,\dots,v_d]$ into $\partial E$ by choosing as its base a regular $(d-1)$-simplex with vertices in $\partial E\cap H(e_d,(3(d-1))^{-2})$ and the origin as its apex. 

Let $T_y(K)$ be the tangent space (or supporting hyperplane) of $K$ at the point $y$. It can be identified with $\IR^{d-1}$ having $y$ as its origin. Since $K\in \cC^2_+$, there is a paraboloid $Q_y$ osculating $\bd K$ at $y$ and for any $h>0$ there is a unique affine map $A_y$ such that $A_y(C^{E}(0,1))=C^{Q_y}(y,h)$, while mapping the coordinate axes onto the coordinate axes of $T_y(K)\times \IR$. Moreover, for any $\varepsilon>0$, there is $h_0>0$ such that the mapped cap 
\[
D:=A_y^{-1}(C^{K}(y,h))
\]
has Hausdorff distance at most $\varepsilon$ from the cap $C^{E}(0,1)$, uniformly in $y\in \partial K$ and $h\le h_0$.

As in \cite{TuWe} we define sets $\tilde C^0,\dots,\tilde C^d\subset C^{\partial K}(y,h)$ such that the statements \eqref{eq:area-smallcaps} and \eqref{eq:simplex-cone-affine} hold for any choice of $x_0\in \tilde C^0,\dots,x_d\in \tilde C^d$. Further, for every $\eta>0$, we define sets $C^i, D^i$ by
\begin{equation} \label{eq:eta}
D^i:=A_y^{-1}(\tilde C^i) \cap B(v_i,\eta)
,\
C^i:=A_y(D^i) \subset \tilde C^i,\qquad i=0,\dots,d.
\end{equation}
Given $\tilde{C}^0$ as above, we now prove that \eqref{eq:area-types} holds for sufficiently small $\eta$ depending on $\tilde{C}^0$. It can be deduced from the construction in \cite{TuWe} that there is $\delta>0$ only depending on $K$ such that
\begin{equation} \label{eq:ball-C0}
A_y^{-1}\big(C^{\partial K}(y,h)\big)\cap B(0,2\delta)\subset A_{y}^{-1}(\tilde{C}^0).
\end{equation}

Each of the open polyhedral sets 
\[
S_0(j)=\{z\in \IR^d\setminus S_0\colon \bey(z,S_0)=j\},\qquad j=1,\dots,d-1,
\]
is the union of intersections of affine half-spaces bounded by the affine spans of the facets of $S_0$. Therefore, we find $r_0\in (0,\frac{\delta}{2})$ and centers $z_j\in C^{\partial E}(0,1)\cap B(0,\frac{\delta}{2})$ such that, for any $j=1,\dots,d-1$,
\begin{equation} \label{eq:2r}
B(z_j,2r_0)\subset S_0(j).
\end{equation}
By continuity, we can preserve this property under perturbations of the vertices of $S_0$. More precisely, let $x_0'\in D^0,\dots,x_d'\in D^d$ and consider the simplex $ \Delta_{x'}=[x_0',\dots,x_d']$, which is close in Hausdorff distance to $S_0$ for small $\eta>0$ as in \eqref{eq:eta}. We claim that for $\eta$ small enough it holds that, for all $x_0'\in D^0,\dots,x_d'\in D^d$ and $j=1,\dots,d-1$,
\begin{equation} \label{eq:balls-region}
	B(z_j,r_0)\subset \Delta_{x'}(j)=\{z\in \IR^d\setminus \Delta_{x'}\colon  \bey(z,\Delta_{x'})=j\},
\end{equation}
where $r_0$ is as in \eqref{eq:2r}. To see this, choose some ball $B$ containing a neighborhood of $C^{\partial E}(0,1)$.
Since $S_0$ is a $d$-simplex, the unit normals $n_1,\dots,n_{d+1}$ of the faces of $\Delta_{x'}$ are continuous functions of $x_0',\dots,x_d'$ in a small enough neighborhood of $v_0,\dots,v_d$. More precisely, continuity holds as long as $ \Delta_{x'}=[x_0',\dots,x_d']$ is not degenerate. Moreover, each of the regions $\Delta_{x'}(j)\cap B$ is a continuous function (with respect to Hausdorff distance) of the unit normals $n_1,\dots,n_{d+1}$ in a neighborhood of the normals of $S_0$. Thus, it is possible to choose $\eta>0$ small enough such that \eqref{eq:balls-region} holds for any choice of $x_0'\in B(v_0,\eta),\dots,x_d'\in B(v_d,\eta)$. Finally, we fix $\eta$ and thus the regions $C^0,\dots,C^d$ such that the inclusion \eqref{eq:balls-region} holds.

In particular, it follows from \eqref{eq:balls-region} that, for any choice of $x_0'\in D^0,\dots,x_d'\in D^d$, the ball $B(z_j,r_0)\subset B(0,\delta)$ is contained in $\Delta_{x'}(j)$ and 
\[
\cH^{d-1}(\Delta_{x'}(j)\cap C^{\partial E}(0,1)\cap B(0,\delta))\ge 2c_0,\qquad j=1,\dots,d-1,
\]
where $c_0>0$ depends only on $r_0$. By the approximation above, we can choose $h_0>0$ small enough such that $D$ is sufficiently close to $C^{E}(0,1)$ such that, by \eqref{eq:ball-C0}, 
\begin{equation} \label{eq:region-lower}
\cH^{d-1}\big(\Delta_{x'}(j)\cap A_y^{-1}(\tilde{C}^0)\big)
\ge c_0,\qquad j=1,\dots,d-1,
\end{equation}
uniformly in $y\in \partial K$ and $h\le h_0$. 

Finally, for the proof of \eqref{eq:area-types} choose $x_i\in C^i$ with $x_i'=A_y^{-1}(x_i)\in D^i$. Let $F$ be a face of $\Delta_{x'}$. Then $A_y(F)$ is a face of $\Delta_{x}$ and the map $A_y$ satisfies, for any $z\in A_y^{-1}(C^{\partial K}(y,h))$,
\[
z \text{ is beyond } F \qquad \text{if and only if}\qquad A_y(z) \text{ is beyond }A_y(F).
\]
This gives
\[
\Delta_{x}(j)\cap \tilde{C}^{0}
=A_y\big(\Delta_{x'}(j)\cap A_y^{-1}(\tilde{C}^{0})\big),
\]
and it follows from \eqref{eq:region-lower} that, for some $c>0$,
\begin{equation} \label{eq:regions-lower-affine}
	\cH^{d-1}\big(\Delta_{x}(j)\cap \tilde{C}^{0}\big)
\ge c h^{\frac{d-1}{2}},
\end{equation}
uniformly in $y\in \partial K$ and $h\le h_0$. In the last step we used that due to the curvature bounds on $\partial K$, $ K \in \cC^2_+$, the affine transformation $A_y$ transforms surface areas by multiplication of a factor $h^{\frac{d-1}{2}}$ up to constants depending only on $K$, see \cite[Section~3.1]{TuWe} and the proof of \cite[Lemma~6.2]{RVW08} for details. Combined with Lemma~\ref{lem:cap-convex}, the estimate \eqref{eq:regions-lower-affine} yields \eqref{eq:area-types}. This completes the proof of Proposition~\ref{pro:lower-main}.
\end{proof}

\section{Proof of the central limit theorem}\label{sec:clt}

In this section we prove Theorem~\ref{thm:CLT} using the machinery of normal approximation for stabilizing functionals developed in \cite{LachSchulteYukich}. Let $K\in \cC^2_+$ and $P_n=[X_1,\dots,X_n]$ be as in Theorem~\ref{thm:CLT}. Assume without loss of generality that $\cH^{d-1}(\bd K)=1$. Moreover, assume that $n\ge 9$ since if $d+2\le n< 9$, then the statement of Theorem~\ref{thm:CLT} follows directly from the fact that $\V f_k(P_n)>0$ and $d_K\le 1$.

We can represent the number of $k$-faces of the random polytope $P_n$ as a sum of local contributions as follows. Almost surely, the random polytope $P_n$ is simplicial and 
\begin{equation} \label{eq:score}
f_k(P_n)=\sum_{i=1}^n\xi_k(X_i,\cX_n),\qquad k=1,\dots,d-1,
\end{equation}
where $\cX_n=\{X_1,\dots,X_n\}$ is a random point set and, for any finite $\cX\subset \bd K$ and $x\in \cX$, 
\begin{equation*} 
\xi_k(x,\cX)=\frac{1}{k+1}\sum_{F\in \cF_k([\cX])} \mathbf{1}_{x\in F},\quad x\in \cX,\qquad k=1,\dots,d-1.
\end{equation*}
Only vertices $y\in \cX$ sharing a $k$-face with $x$ contribute to the value of $\xi_k(x,\cX)$ and these vertices are usually close to $x$. A radius of stabilization for $\xi_k$ is a radius $R(x,\cX)$ depending on $x$ and $\cX$ such that points of $\cX$ with distance larger than $R(x,\cX)$ do not contribute to the value of $\xi_k(x,\cX)$. Formally define, for a finite set $\cX\subset \bd K$ and $x\in \bd K$,
\begin{align} \label{eq:rad-stab}
R(x,\cX\cup \{x\})=\inf\{R>0\colon  &K\cap H_F^+\subset B(x,R) \text{ for all }\notag\\
&F\in \cF_{d-1}([\cX,x]) \text{ with }x\in F\},
\end{align}
which is the smallest radius $R$ such that the ball $B(x,R)$ contains all facets of $[\cX,x]$ which arise from adding $x$. Equivalently, it is the distance of the farthest point from $x$ in any such facet. Then $R$ is a radius of stabilization for $\xi_k$, i.e. 
\[
\xi_k(x,\cX\cup\{x\})
=\xi_k(x,\cX\cap B(x,R(x,\cX))), \qquad k=1,\dots,d-1.
\]

The following theorem shows that $R$ as given in \eqref{eq:rad-stab} satisfies an exponential deviation bound in case of binomial input (the case of Poisson input is very similar and discussed in Section~\ref{sec:poisson}).

\begin{theorem}\label{th:rad_influence}
There exist $c,C>0$ such that for each $x \in \bd K$ and $r>0$ we have
$$
\P( R(x, \cX_{n} \cup\{x\}) \geq r)
\leq
C \exp(- c r^{d-1} n),\qquad n\in \IN. 
$$
\end{theorem}
\begin{proof}
	Recall Blaschke's rolling theorem in the form \eqref{eq:in-out} with a ball of radius $r_{in}$ rolling inside and a ball of radius $r_{out}$ rolling outside. Since $R\le {\rm diam}(K)$, it is sufficient to show the statement of Theorem~\ref{th:rad_influence} for $r\le r_{in}$.

	Let $x\in \bd K$ and $r\le r_{in}$. If $R(x, \cX_n\cup \{x\}) \geq r$, then there is a hyperplane $H_F$ containing a facet $F \in \cF_{d-1}([\cX_n,x])$ with $x$ as one of its vertices such that the diameter of $H_F \cap K$ exceeds $r$. Thus, also the diameter of $H_F \cap B(x-r_{out} u_x, r_{out})$ exceeds $r$. Hence the diameter of the ball $H_F \cap B(x-r_{in} u_x, r_{in})$ exceeds $\frac{r_{in}}{r_{out}} r$, and its radius exceeds $\frac{r_{in}}{2r_{out}} r$. 
Because of
$$ x \in H_F \cap B(x-r_{in} u_x, r_{in}) \subset H_F \cap K , $$
there is a ball with center $\tilde x \in \bd K$, of radius $\frac{r_{in}}{2r_{out}} r$, which contains $x$ in its boundary, such that 
\begin{equation}\label{eq:cap_ball}
B\left(\tilde x, \frac{r_{in}}{2r_{out}} r \right) \cap \bd K \subset H_F^+ \cap \bd K 
.
\end{equation}
We use the following geometric lemma.
\begin{lemma}\label{le:points_cover}
For any $x  \in \bd K$ there are $z_1, \dots , z_{2^{d-1}} \in \bd K$ such that the following holds:
if $\tilde x \in \bd K$ with $\|\tilde x -  x\| = \rho < r_{in}$, then there is an $i \in \{1, \dots , 2^{d-1}\}$ such that 
$$ B\left(z_i, \frac 1{4d^2} \rho \right) \subset {\rm int}\, B(\tilde x, \rho) . $$ 
\end{lemma}
\begin{proof}
Without loss of generality assume that $x$ is the origin, and that the outer unit normal vector of $x$ equals $-e_d$. The orthogonal hyperplanes to the coordinate vectors $e_1, \dots, e_{d}$ dissect $\R^d$ into $2^d$ orthants from which those which contain $e_d$ may have non-empty intersection with the interior of $K$, say $O_1, \dots,  O_{2^{d-1}}$.
Assume that $O_1 = \{x \in \R^d \colon x_1, \dots, x_d \geq 0 \}$, and that $\tilde x \in O_1$. 

We define
$$z_1 = \left(\frac \r {\sqrt {d(d-1)}}, \dots, \frac \rho {\sqrt {d(d-1)}} , h_d \right)    , $$
where $h_d\ge 0$ is chosen such that $z_1 \in \bd K $ (and $z_2, \dots , z_{2^{d-1}}$ analogously).
Because $\tilde x$ and $z_1$ are not in the interior of the ball with radius $r_{in}$, we have
$$
h_d \leq \frac {\r^2}{d r_{in}} \leq \frac {\r}{d }
, \text{ and }\ 
\tilde x_d \leq \frac {\rho^2}{2 r_{in}} \leq \frac {\rho}{2} 
$$
We estimate the distance between $z_1$ and $\tilde x$.
\begin{align*}
\| \tilde x -z_1  \|^2
&=
\| \tilde x \| ^2  -  2 \tilde x \cdot  z_{1}  +  \|z_{1}\|^2
\\ & \leq 
\rho^2   -   \frac{2\rho}{\sqrt{d(d-1)}} \sum_{i=1}^d \tilde x_i + \frac{2 \rho}{\sqrt{d(d-1)}} \tilde x_d - 2 h_d \tilde x_d  +  \rho^2 \frac 1d  + h_d^2
\\ & \leq
\rho^2 \left(1 + \frac 1 d - \frac{2}{\sqrt{d(d-1)}}\right) + \frac{2\rho }{\sqrt{d(d-1)}} \tilde x_d - 2 \tilde x_d h_d  + h_d^2
\end{align*}
because $\sum _1^d \tilde x_i \geq \| \tilde x \|_2 = \rho$ for $\tilde x_i \geq 0$.
The function $ \frac{2\rho }{\sqrt{d(d-1)}} \tilde x_d - 2 \tilde x_d h_d  + h_d^2 $ has no maximum in the interior of the region 
$(\tilde x_d, h_d) \in[0, \frac \rho 2] \times[0, \frac \rho d]$, and attains for $\tilde x_d= \frac \rho 2$, $h_d = 0$ its maximum $\frac{\rho ^2}{\sqrt{d(d-1)}}$. Hence
\begin{align*}
\| \tilde x -z_1  \|^2
& \leq
\rho^2 \left(1 + \frac 1 d - \frac{1}{\sqrt{d(d-1)}} \right)\\
& \leq  
\rho^2 \left(1 - \frac{1}{2d^2} \right)
\end{align*}
since $(1-\frac{1}{d})^{-1/2}\ge 1+\frac{1}{2d}$.
This proves that the ball with center $z_1$ of radius
$$
\rho- \sqrt{ \rho^2 - \frac{\rho^2}{2d^2}  } > \frac 1{4d^2} \rho
$$
is contained in the interior of $B(\tilde x, \rho)$. This completes the proof of Lemma~\ref{le:points_cover}.
\end{proof}

From \eqref{eq:cap_ball} and Lemma \ref{le:points_cover} we deduce that for each boundary point $x$ and for any fixed $r$ there are points $z_1, \dots, z_{2^{d-1}}$ such that any cap of diameter at least $r$ contains one of the balls
$$
B \left(z_i, \frac{r_{in}}{8 d^2 r_{out}} r \right),\qquad i=1,\dots,2^{d-1}.
$$
This implies
\begin{align*}
\P( R(x, \cX_{n} \cap\{x\}) \geq r)
& =
\P\big( \exists F \in \cF([x,\cX_{n}])\colon x \in F,\, {\rm diam}( H_F\cap K ) \geq r \big)
\\ & \leq
\sum_{i=1}^{2^{d-1}} \P  \left( B\left(z_i, \frac{r_{in}}{8 d^2 r_{out}} r \right) \cap \cX_{n} = \emptyset \right)
\\ & \leq
\sum_{i=1}^{2^{d-1}} \left(1- \cH^{d-1} \left(B\left(z_i, \frac{r_{in}}{8 d^2 r_{out}} r \right) \cap \bd K\right) \right)^{n}.
\end{align*}

With $c'=(\frac{r_{in}}{8 d^2 r_{out}})^{d-1}c$, where $c$ is from Lemma \ref{lem:cap-convex}, it holds that
\begin{align*}
\P( R(x, \cX_{n} \cap\{x\}) \geq r)
& \leq
2^{d-1} \left(1- c' r^{d-1} \right)^{n} 
\\ & \leq
2^{d-1} \exp(- c' r^{d-1} n).
\end{align*}
This completes the proof of Theorem~\ref{th:rad_influence}.
\end{proof}

It follows from Theorem~\ref{th:rad_influence} that the scores $\xi_k$ have bounded moments in the sense of \cite{LachSchulteYukich}.

\begin{proposition}\label{pro:moments}
Let $k=0,\dots,d-1$ and $q\in (0,\infty)$. There exists $C_q>0$ such that for any $\cY\subset \bd K$ with $|\cY|\le 7$, $x\in \partial K$ and $n\ge 9$ we have
\[
\IE\, \xi_k(x,\cX_{n-8}\cup \{x\}\cup \cY)^{q} \le C_q.
\]
\end{proposition}
\begin{proof}
	Let $k=1,\dots,d-1$, $x\in \partial K$ and $\cY\subset \partial K$ with $|\cY|\le 7$. If a $k$-face of the convex hull $Q_n=[\cX_{n-8},x,\cY]$ contains $x$, then it is a subset of a facet contained in a ball around $x$ with radius equal to the radius of stabilization. Therefore,
	\begin{equation}\label{eq:kface-influence}
	\xi_k(x,\cX_{n-8}\cup \{x\}\cup \cY)
	\le \frac{N_k}{k+1},
	\end{equation}
	where $N_k$ is the number of $k$-faces of $Q_n$ contained in the ball $B(x,R_n)$ with random radius $R_n=R(x,\cX_{n-8}\cup \{x\}\cup \cY)$. Analogously, define $N_0$ as the number of vertices of $Q_n$ which belong to this ball.  Almost surely, the points in $\cX_{n-8}$ are affinely independent and every $k$-face of $[\cX_{n-8},x,\cY]$ has $\ell$ vertices, where $k+1\le \ell\le k+9$. Thus, assuming that $N_0\ge d+9$, almost surely, 
	\begin{equation} \label{eq:binom}
	N_k
	\le \binom{ N_0 }{k+1}+\cdots+\binom{ N_0 }{k+9}
	\le \Big(\frac{e N_0}{k+9}\Big)^{k+9}.
	\end{equation}

	Let $A_n=\cH^{d-1}(B(x,R_n)\cap \partial K)$ and $t>0$. Then
	\begin{equation}\label{eq:tails}
	\IP(N_0\ge t)
	= \IP\Big(N_0\ge t, A_n\ge \frac{t}{n}\Big)
	+ \IP\Big(N_0\ge t, A_n\le \frac{t}{n}\Big).
	\end{equation}
	Moreover, by Lemma \ref{lem:cap-convex}, we have $A_n\le C R_n^{d-1}$. By Theorem~\ref{th:rad_influence} and the fact that $R(x,\cX)$ does not increase if points are added to $\cX$, it holds that
	\begin{align}
	\IP\Big(N_0\ge t, A_n\ge \frac{t}{n}\Big)\notag
	&\le \IP\Big(R_n^{d-1}\ge \frac{t}{C n}\Big)\\ \notag
	&\le \IP\Big(R(x,\cX_{n-8}\cup \{x\})^{d-1}\ge \frac{t}{C n}\Big)\\  
	&\le Ce^{-ct}.\label{eq:tail-1}
	\end{align}
	Choose a deterministic radius $R_{n,t}(x)$ such that 
	\[
	\cH^{d-1}\big(B(x,R_{n,t}(x))\cap \bd K \big)=\frac{t}{n},\qquad t\le n.
	\]
	(Recall that we assume $\cH^{d-1}(\partial K)=1$.)
	If $N_0\ge t>9$ and $A_n\le \frac{t}{n}$, then the ball $B(x,R_{n,t}(x))$ contains at least $t$ points of $\cX_{n-8}\cup \{x\}\cup \cY$ and thus at least $t-8$ points of $\cX_{n-8}$. Therefore, for $t>9$,
	\begin{equation}\label{eq:tail-2}
	\IP\Big(N_0\ge t, A_n\le \frac{t}{n}\Big)
	\le \IP\big(|\cX_{n-8}\cap B(x,R_{n,t}(x))|\ge t-8\big)
	\le Ce^{-ct},
	\end{equation}
	by concentration bounds for sums of independent Bernoulli random variables. If $t>n$, then the bound trivially holds. Combining \eqref{eq:tails}, \eqref{eq:tail-1} and \eqref{eq:tail-2}, the tails of $N_0$ decay exponentially, uniformly in $n,x$ and $Y$. Thus, by \eqref{eq:kface-influence} and \eqref{eq:binom}, the tails of $\xi_k(x,\cX_{n-8}\cup \{x\}\cup \cY)$ decay sufficiently fast for the required moment bounds to hold. This completes the proof of Proposition~\ref{pro:moments}.
\end{proof}

\begin{proof}[Proof of Theorem~\ref{thm:CLT}]
	Let $k\in \{1,\dots,d-1\}$. By \eqref{eq:score}, it is sufficient to prove the statement of Theorem~\ref{thm:CLT} for 
\begin{equation} \label{eq:Hn}
H_n'
=\sum_{i=1}^{n}\xi_k(X_i,\cX_n)
\end{equation}
as in \eqref{eq:score}. By Theorem~\ref{thm:var} it holds that
\begin{equation} \label{eq:H-var}
\V H_n'
=\V f_k(P_n)
\ge c n.
\end{equation}

In the setting of \cite{LachSchulteYukich} consider the measure space $(\partial K,\cB\cap \bd K,\cH^{d-1})$ equipped with the Euclidean metric, where $\cB$ is the Borel $\sigma$-algebra on $\IR^d$. By Lemma~\ref{lem:surface-balls}, the condition in \cite[eq.~(2.1)]{LachSchulteYukich} is satisfied with $\gamma=d-1$. 

We verify the assumptions of \cite[Corollary 2.2(b)]{LachSchulteYukich} for the choice of $H_n'$ as in \eqref{eq:Hn}. By Theorem~\ref{th:rad_influence} the functionals $\xi_k$ are exponentially stabilizing in the sense of \cite{LachSchulteYukich}. By Proposition~\ref{pro:moments} they satisfy the moment condition in \cite[Theorem~2.1(b)]{LachSchulteYukich} for any $p\in (0,1]$. The assumption of exponential decay with respect to a set is trivially satisfied with $I_{K,n}=n$ as in \cite[eq.~(2.13)]{LachSchulteYukich}. Finally, from the lower bound in \eqref{eq:H-var} we deduce that $\sup_{n\ge 1}I_{K,n}/\V H_n' \le c^{-1}$. Thus, all the assumptions of \cite[Corollary 2.2(b)]{LachSchulteYukich} are satisfied and Theorem~\ref{thm:CLT} follows. 
\end{proof}
\section{Poisson Boundary Polytopes}\label{sec:poisson}

In this section we prove Theorem~\ref{thm:poisson}. Let $N \sim \pi(t)$ be Poisson distributed with parameter $t>1$. Choose $N$ independent random points uniformly on $\bd K$, and denote the convex hull of these random points by $P_N$. We call $P_N$ a Poisson boundary polytope. 

First, we prove that the variance lower bound carries over to $\V f_k(P_N)$. Analogously to \eqref{eq:var-lower} and using Theorem \ref{thm:var} we have
$$
\V f_k(P_N) \geq \E[\V ( f_k (P_N)|N)] \geq \E c N = c t.
$$
An upper bound for the variance follows immediately from Theorem 10 in \cite{Re03} and the Poincar\'e inequality for Poisson random variables, which in our context says that
$$
\V f_k(P_N) \leq t \E (f_k(P_{N+1})-f_k(P_N))^2 \leq c t. 
$$
(For the general Poincar\'e inequality for functionals of Poisson random measures see \cite[Theorem 18.7]{LP}.) This proves \eqref{eq:Poisson-Var}.

For the central limit theorem \eqref{eq:Poisson-CLT} observe that the method leading to the central limit theorem in Theorem \ref{thm:CLT} works with some immediate changes also in our case. The radius of stabilization is the same, and the work \cite[Cor.2.2]{LachSchulteYukich} contains both the CLT for binomial input (the case considered in Sections~\ref{sec:var} and \ref{sec:clt}) and the analogous CLT for Poisson input we consider in this section. The only changes consist in changing the binomial distribution to the Poisson distribution at several places which is immediate.

\subsection*{Acknowledgements}

We would like to thank Martina Juhnke-Kubitzke for pointing us to the notion of combinatorial types. Further, we thank the Hausdorff Research Institute for Mathematics, University of Bonn, for providing an
excellent working environment during the Dual Trimester Program \textit{Synergies between modern probability, geometric analysis and stochastic geometry} where this work originated.

The research of MR was funded in part by the Deutsche Forschungsgemeinschaft (DFG, German Research Foundation) -- SPP 2265, project 531562368. The research of MS was funded in whole or in part by the Austrian Science Fund (FWF) [Grant DOI: 10.55776/J4777]. For open access purposes, the authors have applied a CC BY public copyright license to any author-accepted manuscript version arising from this submission. 

\bibliographystyle{abbrv}
\bibliography{cltfaces}

\end{document}